\documentclass[12pt]{amsart}
\usepackage{graphicx}
\usepackage{amsmath}
\usepackage{amssymb}
\usepackage{amsfonts}
\usepackage{verbatim}
\usepackage{scalefnt}
\usepackage{multirow}
\usepackage{enumerate}
\usepackage{mathtools}
\usepackage{float}
\usepackage{enumitem}
\restylefloat{figure}
\usepackage[margin=3cm]{geometry}

\usepackage{fancyhdr}
\usepackage{hyperref}

\newtheorem{theorem}{Theorem}
\newtheorem{lemma}[theorem]{Lemma}

\newtheorem*{definition*}{Definition}

\newtheorem{remark}[theorem]{Remark}

\numberwithin{equation}{section}
\numberwithin{theorem}{section}


\title[The counting version of a problem of Erd\H{o}s]%
  {The counting version of a problem of Erd\H{o}s}

\author{P\'eter P\'al Pach}
\email{ppp@cs.bme.hu}
\address{MTA-BME Lend\"ulet Arithmetic Combinatorics Research Group, Department of Computer Science and Information Theory, Budapest
  University of Technology and Economics, 1117 Budapest, Magyar tud\'osok
  k\"or\'utja 2., Hungary}

\author{Rich\'ard Palincza}
\email{pricsi@cs.bme.hu}
\address{MTA-BME Lend\"ulet Arithmetic Combinatorics Research Group, Department of Computer Science and Information Theory, Budapest
  University of Technology and Economics, 1117 Budapest, Magyar tud\'osok
  k\"or\'utja 2., Hungary}

\thanks{}

\begin{document}

\begin{abstract}

A set $A$ of natural numbers possesses property $\mathcal{P}_h$, if there are no distinct elements $a_0,a_1,\dots ,a_h\in A$ with $a_0$ dividing the product $a_1a_2\dots a_h$. Erd\H{o}s  determined the maximum size of a subset of $\{1,\ldots, n\}$ possessing property $\mathcal{P}_2$. More recently, Chan, Gy\H{ori} and S\'ark\"ozy \cite{CGS} solved the case $h=3$, finally the general case  also got resolved by Chan \cite{Chan}, the maximum size is $\pi(n)+\Theta_h(\frac{n^{2/(h+1)}}{(\log n)^{2}})$. 

In this note we consider the counting version of this problem and show that the number of subsets of $\{1,\ldots, n\}$ possessing property $\mathcal{P}_2$ is $T(n)\cdot e^{\Theta(n^{2/3}/\log n)}$
 for a certain function $T(n)\approx (3.517\dots)^{\pi(n)}$. For $h>2$ we prove that the number of subsets possessing property  $\mathcal{P}_h$ is $T(n)\cdot e^{\sqrt{n}(1+o(1))}$.  
 
This is a rare example in which the order of magnitude of the lower order term in the exponent is also determined.

\end{abstract}

\date{\today}
\maketitle

\section{Introduction}
We say that a set $A\subseteq \mathbb{N}$  \emph{possesses property $\mathcal{P}_h$}, if there are no distinct elements $a_0,a_1,\dots ,a_h\in A$ with $a_0$ dividing the product $a_1a_2\dots a_h$. Let us denote  the set of those subsets of a set $S$ that possess property $\mathcal{P}_h$ by $\mathcal{P}_h(S)$.  

Property $\mathcal{P}_h$ was introduced by  Erd\H{o}s~\cite{Erd38} back in 1938, who studied the maximum size of a subset of $[n]:=\{1,\ldots, n\}$ possessing property $\mathcal{P}_2$. He proved that the extremal size is $\pi(n)+\Theta(\frac{n^{2/3}}{(\log n)^{2}})$, that is, besides showing that the main term is $\pi(n)$, he could determine the lower order term up to a constant factor.

More recently, Chan, Gy\H{o}ri and S\'ark\"ozy \cite{CGS} determined the lower order term for $h=3$ too, then finally Chan \cite{Chan} resolved the general case $h>1$ by showing that the extremal size is $\pi(n)+\Theta_h(\frac{n^{2/(h+1)}}{(\log n)^{2}})$. Chan even studied the dependence of the constant on $h$ hidden in the $\Theta_h$ notion (assuming $n$ is sufficiently larger than $h$). This dependence was determined (up to a constant factor) by S\'andor and the first author \cite{PS} establishing that for $n$ sufficiently larger than $h$ the extremal size is $\pi(n)+\Theta(\frac{n^{2/(h+1)}}{(\log n)^{2}})$.

Given now the satisfying answer on how large a subset of $[n]$ possessing property $\mathcal{P}_h$ could be, a natural next step is to estimate how many subsets of $[n]$ possess property $\mathcal{P}_h$, that is, how large $\mathcal{P}_h([n])$ is. This is the question we are concerned about in this paper.

Indeed, enumerating subsets of $[n]$ satisfying various properties was initiated by Cameron and Erd\H{o}s~\cite{CamErd} in the 80s. In particular they considered the enumeration problem of primitive sets (note that a set is primitive if it possesses property $\mathcal{P}_1$). The nature of this problem, and also the applied techniques, obtained results are different from the case of $\mathcal{P}_h$ with $h\geq 2$. For more on the case of primitive sets we refer to the papers  \cite{Angelo,LPP,McNew, Vijay}.

Another related question is enumerating multiplicative Sidon subsets of $[n]$, also initiated by Cameron and Erd\H{o}s. Recently, Liu and the first author \cite{LP} proved that the number of multiplicative Sidon subsets of $[n]$ is $R(n)\cdot 2^{\Theta(\frac{n^{3/4}}{(\log n)^{3/2}})}$ for a certain function $R(n)\approx 2^{1.815\pi(n)}$ which they specified. That is, the order of magnitude of the lower order term in the exponent is also determined.

\subsection{Main result}

Let $H_h(n) := |\mathcal{P}_h([n])|$ denote the number of those subsets of $[n]$ that possess property $\mathcal{P}_h$.

\begin{theorem}\label{count}
There exist positive constants $c_1$ and $c_2$ such that, for the number of those subsets of $[n]$ that possess property $\mathcal{P}_2$, we have
$$T(n)\cdot e^{c_1n^{2/3}/\log n}  \leq H_2(n)\leq T(n)\cdot e^{c_2n^{2/3}/\log n},$$
if $n$ is sufficiently large.

Let $h\geq 3$ be an integer. 
For large enough $n$, the number of subsets of
$[n]$ possessing property $\mathcal{P}_h$ satisfies
$$T(n)\cdot e^{\sqrt{n}}e^{-11\sqrt{n}\log\log n/\log n}  \leq H_h(n)\leq T(n)\cdot e^{\sqrt{n}}e^{4\sqrt{n}\log\log n/\log n},$$
where
$$T(n):=\prod\limits_{\substack{\sqrt{n}<p\leq n, \\ p \text{ prime} }}  ([n/p]+1).$$

\end{theorem}

A more explicit formula for the function $T(n)$ is 
$$T(n)=e^{O(n^{1/2})}\cdot \prod\limits_{i=1}^{n^{1/2}}(1+1/i)^{\pi(n/i)}.$$ 
 A more crude estimate is $T(n)=({\alpha}+o(1))^{\pi(n)}$, where
\begin{equation*}
	\alpha:=\prod\limits_{i=1}^\infty(1+1/i)^{1/i}=3.517\dots
\end{equation*}

Theorem~\ref{count} is another rare example of an enumeration result in which the correct order of magnitude of the lower order term is given. Moreover, in the case $h>2$ even the lower order term (in the exponent) is determined up to a $1+o(\log\log n/\log n)$ factor.

\subsection{Related results}
The past decade has witnessed rapid development in enumeration problems in combinatorics. In particular, a  related problem of enumerating \emph{additive} Sidon sets, i.e.~sets with distinct sums of pairs, and its generalisation to the so-called $B_h$-sets was studied by Dellamonica, Kohayakawa, Lee, R\"odl and Samotij~\cite{DKLRS16,DKLRS18,KLRS15}. For more recent results on enumerating sets with additive constraints, see e.g.~\cite{BLSh17, BLShT15, BLShT18, Gre04, Hancock-Staden-Treglown,Sap03,Tran}. Many of these counting results use the theory of hypergraph containers introduced by Balogh, Morris and Samotij~\cite{BMS15}, and independently by Saxton and Thomason~\cite{ST}. We refer the readers to~\cite{BMS15,ST} for more literature on enumeration problems on graphs and other settings.

\medskip

\noindent\textbf{Organisation of the paper.} Section~\ref{sec-prelim} sets up notations and tools needed for the proof. In Section~\ref{sec-pf}, we prove Theorem~\ref{count}. Some concluding remarks are given in Section~\ref{sec-conclude}.

\medskip

\noindent\textbf{Asymptotic notations.} Throughout the paper we will use the standard notation $\ll$, $\gg$ and respectively $O$ and $\Omega$ is applied to positive quantities in the usual way.
That is, $X \gg Y$, $Y \ll X$, $X = \Omega(Y )$ and $Y = O(X)$ all mean that $X \geq cY$,
for some absolute constant $c > 0$. If both $X \ll Y$ and $Y \ll X$ hold we write
$ X = \Theta(Y )$. If the constant $c$ depends on a quantity
$h$, we write $X \gg_h Y$, $Y = \Omega_h(Y )$, and so on.

\section{Preliminary lemmas}\label{sec-prelim}

Throughout the paper we will need some bounds on the prime-counting function $\pi(x).$ The following standard bound will be enough for our purposes:

\begin{lemma}\label{lem-prime-count} If $x$ is sufficiently large, then
$$\frac{x}{\log x}+\frac{x}{(\log x)^2}\leq \pi(x) \leq \frac{x}{\log x}+\frac{2x}{(\log x)^2}.$$
\end{lemma}

For proving Theorem~\ref{count} we will also use multiplicative bases and   two lemmas from \cite{PS}.

We say that 
the set $B \subseteq \mathbb{Z}^+$ forms a {\it multiplicative basis of order $h$} of a set $S$, if every element $ s \in S$ can
be written as the product of $h$ members of $B$. In particular, $B$ is a multiplicative basis of order $h$ for $[n]$ if $[n]\subseteq B^h$, that is, if each positive integer up to $n$ can be expressed as a product of $h$ (not necessarily distinct) elements of $B$.

We will consider multiplicative bases of minimum size, the existence of such ``small'' bases is guaranteed by the following lemma.

\begin{lemma}\label{MB}
Let $h\geq 2$ be an integer. There exists a multiplicative basis $B$ of order $h$ for $[n]$ of size
$$|B|=\pi(n)+O_h\left(\frac{n^{2/(h+1)}}{(\log n)^2}\right).$$
\end{lemma}

\begin{proof}
This follows from Theorem~1 of \cite{PS}.
\end{proof}

Finally, the lemma below describes a connection between a set possessing property $\mathcal{P}_h$ and a multiplicative basis of order $h$.

\begin{lemma}\label{inj}
Let $A\subseteq [n]$ be a set possessing property $\mathcal{P}_h$ and  $B\subset [n]$ be a multiplicative basis of order $h$ for $[n]$.
Then there exists an injective mapping $\varphi: A\to B$ such that for $\varphi(a)=b$ there exist integers $b_2,\dots ,b_h\in B$ such that $a=bb_2\dots b_h$.
\end{lemma}

\begin{proof}
This is a special case of Lemma~12 in \cite{PS}.
\end{proof}

\section{Proof of Theorem~\ref{count}}\label{sec-pf}

\subsection{}
First we consider the case $h=2.$

\subsubsection{Lower bound.}
For obtaining the lower bound we will use linear hypergraphs. Let us recall that a  hypergraph is linear if each pair of hyperedges intersects in at most one vertex.

Let us consider sets of the form $A=A_1\cup A_2\subseteq [n]$, where:
\begin{itemize}
    \item In $A_1$ each element has a prime factor from $(\sqrt{n},n]$ and for every prime $p\in (\sqrt{n},n]$ the number of multiples of $p$ contained in $A_1$ is at most one.
    \item The set $A_2$ can be obtained in the following way. We take a 3-uniform linear hypergraph $G$ with vertex set $V=\{p:\text{$p$ is a prime},p\in(n^{1/3}/2, n^{1/3})\}$ and edge set $E$. Let $A_2=\{pqr:\{p,q,r\}\in E\}$. That is, $A_2$ contains integers that can be written as a product of three distinct primes such that these three primes form a hyperedge of $G$.
\end{itemize}

We claim that $A=A_1\cup A_2$ possesses property $\mathcal{P}_2$. If $a_0\in A_1$, then there is a prime $p\in (\sqrt{n},n]$ such that $p\mid a_0 $ and $p\nmid a$, if $a\in A\setminus \{a_0\}$, so $a_0\nmid a_1a_2$. If $a_0\in A_2$, then $a_0=pqr$ for three primes $p,q,r\in (n^{1/3}/2,n^{1/3})$. Assume that $a_0\mid a_1a_2$ for some $a_1,a_2\in A\setminus \{a_0\}$. We can assume that $a_1$ is divisible by at least two of the primes $p,q,r$, for instance, $pq\mid a_1$. If $a_1\in A_2$, then this contradicts the linearity of $G$. If $a_1\in A_1$, then $a_1$ has one prime factor larger than $\sqrt{n}$ and two prime factors ($p$ and $q$) from $(n^{1/3}/2,n^{1/3})$, which is a contradiction again.

Hence, each set that can be obtained as $A=A_1\cup A_2$ possesses property $\mathcal{P}_2$. Note that for different pairs $(A_1,A_2)$ we get different sets $A=A_1\cup A_2$.

The number of choices for $A_1$ is 
$$T(n)=\prod\limits_{\substack{\sqrt{n}<p\leq n, \\ p \text{ prime} }} (\lfloor n/p \rfloor +1),$$
since for each prime $p\in (\sqrt{n},n]$ we can include in $A_1$ either one of the $[n/p]$ multiples of $p$ (up to $n$) or none of them, the choices for different primes are independent from each other.

The number of choices for $A_2$ is the number of linear 3-uniform hypergraphs on  vertex set $V$. By dropping out at most three elements from $V$ we get a set of cardinality $|V'|\equiv 1,3\pmod{6}$.  
The number of linear 3-uniform hypergraphs on $V$ is at least the number of Steiner Triple Systems on $V'$ which is known \cite{Wilson}  to be $2^{\Theta(|V'|^2\log|V'|)}$.

Hence, the following lower bound is obtained:
$$\prod\limits_{\substack{\sqrt{n}<p\leq n, \\ p\text{ prime} }} (\lfloor n/p \rfloor +1)\cdot e^{\Theta(n^{2/3}/\log n)}\leq H_2(n).$$

\subsubsection{Upper bound.}

Now, we continue with the upper bound. 
According to Lemma~\ref{MB} there exists a multiplicative basis $B=P\cup X$ of order 2, where $$P=\{p: p\text{ is a prime},p\in (\sqrt{n},n]\}$$
and 
$$|X|\ll n^{2/3}/(\log n)^2.$$
(Note that a multiplicative basis for $[n]$ must contain all the primes up to $n$, therefore the above $P$ is a subset of any multiplicative basis $B$, and we can set $X:=B\setminus P$ for a multiplicative basis $B$ of minimum size.)

Also, by Lemma~\ref{inj}, if $A$ possesses $\mathcal{P}_2$, then there is an injective mapping $\varphi:A\to B$, such that for any $\varphi(a)=b$ we have $b\mid a$. Let $A_P,$ resp. $A_X$, be the set of elements mapped (by $\varphi$) to $P$, resp. $X$:
$$A_P=\varphi^{-1}(P), \quad A_X=\varphi^{-1}(X).$$

The number of choices for $A_P$ is at most
\begin{equation*}
\prod\limits_{\substack{\sqrt{n}<p\leq n, \\ p\text { prime}}} (\lfloor n/p \rfloor +1)=T(n).
\end{equation*}

As $|A_X|= |X|\ll n^{2/3}/(\log n)^2$, the number of choices for $A_X$ is at most $2^{O(n^{2/3}/\log n)}$, therefore, the number of choices for $A=A_P\cup A_X$ is at most 
$$T(n)\cdot e^{O(n^{2/3}/\log n)},$$
as needed.

\subsection{}
Now, let $h\geq 3$ be any integer.

\subsubsection{Lower bound.}

Let us consider sets $A$ where each element has {\it exactly one} prime factor from $(\frac{\sqrt{n}}{\log n},n]$ and for every prime $p\in (\frac{\sqrt{n}}{\log n},n]$ the number of multiples of $p$ contained in $A$ is at most one.

Note that these sets satisfy the required property for {\it every} $h$. Indeed, let $a_0,a_1,\dots,a_h\in A$ be distinct, then there exists a prime $p\in (\frac{\sqrt{n}}{\log n},n]$ which divides $a_0$ and does not divide any of $a_1,a_2,\dots,a_h$, thus $a_0\nmid a_1a_2\dots a_h$.

Let us give a lower bound on the number of these sets.

If $p\in [{\sqrt{n}}{\log n},n]$, then the number of choices (for the multiple of $p$) is $[n/p]+1$ (it is also possible that none of the multiples of $p$ is chosen to be in $A$). These can be chosen independently, since none of them is divisible by any other prime from $(\frac{\sqrt{n}}{\log n},n]$.

Now, if $p\in (\frac{\sqrt{n}}{\log n},{\sqrt{n}}{\log n})$, then we have to exclude those multiples of $p$ that have another prime factor $q\in (\frac{\sqrt{n}}{\log n},{\sqrt{n}}{\log n})$.

Note that $q\leq n/p$. That is, the number of choices (for sufficiently large $n$) is at least \begin{equation*}[n/p]+1-\sum\limits_{\substack{\frac{\sqrt{n}}{\log n}<q\leq n/p, \\ q \text{ prime}}}
\frac{[n/p]}{q}\geq \left([n/p]+1\right) \left(1- \frac{5\log\log n}{\log n}\right),
\end{equation*}
since 
$$\sum\limits_{\substack{\frac{\sqrt{n}}{\log n}<q\leq n/p, \\ q \text{ prime}}} \frac{1}{q}\leq \sum\limits_{\substack{\frac{\sqrt{n}}{\log n}<q\leq {\sqrt{n}}{\log n},  \\ q \text{ prime}} } \frac{1}{q}\leq \frac{5\log\log n}{\log n},$$
according to Mertens' theorem which states that for some constant $M>0$ we have 
$$\sum\limits_{\substack{q\leq x, \\ q\text{ prime}}}\frac1q=\log\log x + M+O\left(\frac{1}{\log x} \right).$$

The choices are independent from each other.

Hence, the number of choices for the set $A$ (for sufficiently large $n$) is at least
\begin{equation}\label{eq7}
\left( 1-\frac{5\log\log n}{\log n} \right)^{\pi(\sqrt{n}\log n)-\pi(\sqrt{n}/\log n)}\prod\limits_{\substack{\sqrt{n}/\log n<p\leq n, \\ p \text{ prime}}} \left([n/p]+1\right).
\end{equation}

Observe that for large enough $n$
\begin{equation}\label{eq1}
1-\frac{5\log\log n}{\log n}\geq e^{-\frac{11\log\log n}{2\log n}},    
\end{equation} 
since for sufficiently small positive $x$ we have $1-x\geq e^{-1.1x}$. Also, 
Lemma~\ref{lem-prime-count} yields (for large enough $n$) that 
\begin{equation}
\label{eq2}
\pi(\sqrt{n}\log n)-\pi(\sqrt{n}/\log n)\leq 2\sqrt{n}.    
\end{equation}
According to \eqref{eq1} and \eqref{eq2} we obtain that
$$\left( 1-\frac{5\log\log n}{\log n} \right)^{\pi(\sqrt{n}\log n)-\pi(\sqrt{n}/\log n)}\geq e^{-\frac{11\sqrt{n}\log\log n}{\log n}}. $$

Thus by using Lemma~\ref{lem-prime-count} again 
we obtain that for large enough $n$ 
$$\prod\limits_{\substack{ \sqrt{n}/\log n<p\leq n, \\ p \text{ prime}}} ([n/p]+1)\geq T(n)\cdot \sqrt{n}^{\pi(\sqrt{n})-\pi(\sqrt{n}/\log n)}\geq T(n)\cdot e^{\sqrt{n}}.$$

Therefore, \eqref{eq7} yields that
$$H_h(n)\geq T(n)e^{\sqrt{n}}e^{-\frac{11\sqrt{n}\log\log n}{\log n}}.$$

\subsubsection{Upper bound.}

According to Lemma~\ref{MB} there exists a multiplicative basis $B=P\cup X$ of order $h$, where $$P=\{p: p\text{ is a prime},p\in (n^{2/(h+1)}/\log n,n]\}$$
and 
$$|X|\ll n^{2/(h+1)}/(\log n)^2.$$

Also, by Lemma~\ref{inj}, if $A\in \mathcal{P}_h([n])$, then there is an injective mapping $\varphi:A\to B$ satisfying that if $\varphi(a)=b$, then $b\mid a$. Let $A_P$, resp. $A_X$, be the set of elements mapped to $P$, resp. $X$:
$$A_P=\varphi^{-1}(P), \quad A_X=\varphi^{-1}(X).$$

The number of choices for $A_P$ is at most \begin{equation}
\label{AP1}    
\prod\limits_{\substack{n^{2/(h+1)}/\log n<p\leq n, \\
 p \text{ prime} }}
 (\lfloor n/p \rfloor +1)=T(n)\cdot \prod\limits_{\substack{n^{2/(h+1)}/\log n<p\leq \sqrt{n}, \\
 p \text{ prime} }} (\lfloor n/p \rfloor +1) .
\end{equation}

Observe that by Lemma~\ref{lem-prime-count} for large enough $n$
\begin{equation}
\label{AP2}
\prod\limits_{\substack{n^{2/(h+1)}/\log n<p\leq \sqrt{n}/\log n, \\
 p \text{ prime} }} (\lfloor n/p \rfloor +1)\leq n^{3\sqrt{n}/(\log n)^2}=e^{3\sqrt{n}/\log n}
\end{equation}
and

\begin{align}
\label{AP3}
\prod\limits_{\substack{\sqrt{n}/\log n< p\leq \sqrt{n}, \\
 p \text{ prime} }} (\lfloor n/p \rfloor +1)&\leq (\sqrt{n}\log n+1)^{\pi(\sqrt{n})}\leq e^{((\log n)/2+\log\log n+1)(2\sqrt{n}/\log n+8\sqrt{n}/(\log n)^2)}\\
&\leq e^{\sqrt{n}}e^{3\sqrt{n}\log\log n/\log n}.
\end{align}

As $|A_X|\leq |X|\ll n^{2/(h+1)}/(\log n)^2$, the number of choices for $A_X$ is at most $2^{O(n^{2/(h+1)}/\log n)}\leq e^{O(\sqrt{n}/\log n)}$, therefore, by \eqref{AP1}, \eqref{AP2} and \eqref{AP3} the number of choices for $A=A_P\cup A_X$ is -- assuming that $n$ is sufficiently large -- at most 
$$T(n)\cdot e^{\sqrt{n}}e^{4\sqrt{n}\log\log n/\log n}.$$

\begin{remark}
The obtained estimation is more precise in the case $h\geq 3$. Let us explain the reason for this, and look at the main difference between the cases $h=2$ and $h\geq 3$. The contribution from the product $\prod (\lfloor n/p\rfloor+1)$, where $p$ is taken from the interval, say, $(t,2t)$ with $t\sim n^{\alpha}$ is $e^{\Theta(n^\alpha)}$. Because of this, in case of $h=2$ we could cut at any $n^\alpha$ with $\alpha\in [1/3,2/3)$ in the definition of $A_1$, since the contribution from the product taken for $p\in(n^{1/3},n^{2/3-\varepsilon})$ is negligible compared to the contribution of $A_2$. To achieve better bounds, one would have to study and understand the count corresponding to the ``$A_2$-part'' better. 

The analogue of $A_2$ could also be considered in the case $h\geq 3$ (by taking $(h+1)$-uniform linear hypergraphs on $V=\{p:\ p\text{ is a prime},\ p\in (n^{1/h}/2,n^{1/h})\}$), but this would only give an $e^{\Theta(n^{2/(h+1)}/\log n)}$ factor. While the term $\Theta(n^{2/(h+1)/\log n})$ in the exponent turns out to be the second order term in the case $h=2$, for $h\geq 3$ it is negligible compared to the additional contribution obtained by cutting lower than $\sqrt{n}$ in the definition of $A_1$. As it turns out from the calculation, we already get the precise lower order term if we cut at $\sqrt{n}/\log n$, though some additional care was needed, as an element of $[n]$ might have more than one prime factors larger than $\sqrt{n}/\log n$. 

\end{remark}

\section{Concluding remarks}\label{sec-conclude}

In this paper, we determine the number of those subsets of $[n]$ that possess property $\mathcal{P}_h$, giving bounds that are optimal up to a constant factor in the exponent of the lower order term $e^{\Theta\left(\frac{n^{2/3}}{\log n}\right)} $ for $h=2$ and $e^{\sqrt{n}(1+o(1))}$ for $h>2$.

A natural extension of these results would be to count those subsets of $[n]$ that satisfy the following property. Let us say that $A$ satisfies property $\mathcal{P}_{r,s}$, if there are no distinct
elements $a_1, a_2, \dots, a_{r+s} \in A$ with $a_1\dots a_r \mid a_{r+1}\dots a_{r+s}$. However, even determining the extremal size for subsets of $[n]$ seem to be difficult. The smallest interesting case is $r=2,s=3$; in this case we know that the largest possible size of a subset of $[n]$ possessing property $\mathcal{P}_{2,3}$ is between $\pi(n)+n^{1/2+o(1)}$ and $\pi(n)+n^{2/3+o(1)}$.

\section{Acknowledgements}

Both authors were supported by the Lend\"ulet program of the Hungarian Academy of Sciences (MTA). PPP was also supported by the National
Research, Development and Innovation Fund (TUDFO/51757/2019-ITM,
Thematic Excellence Program). RP was also supported by  the BME-Artificial Intelligence FIKP grant of EMMI (BME FIKP-MI/SC).

The authors would like to thank  Jaehoon Kim and Hong Liu for pointing out references about counting Steiner Triple Systems.

\end{document}